\documentclass[a4paper,reqno]{amsart}

\usepackage{amsthm, amssymb} \usepackage{amsfonts}
\usepackage{amsmath} \usepackage{mathrsfs} \usepackage[sans]{dsfont}
\usepackage[colorlinks=true,citecolor=black,linkcolor=black,
backref=page]{hyperref}

\usepackage{etoolbox} \makeatletter
\patchcmd{\BR@backref}{\newblock}{\newblock[cited on p.~}{}{}
\patchcmd{\BR@backref}{\par}{]\par}{}{} \makeatother

\usepackage[all]{hypcap} \usepackage{mathtools}
\usepackage{thmtools,cleveref}

\usepackage{kantlipsum} 

\usepackage[normalem]{ulem} 

\usepackage{enumitem}

\declaretheoremstyle[headfont=\bfseries, headpunct={:},
notefont=\bfseries, notebraces={}{} ]{hypstyle}
\declaretheorem[style=hypstyle, name={}]{hyp}

\theoremstyle{theorem} 
\declaretheorem[name=Theorem]{thm} 
\declaretheorem[name=Proposition]{prop}
\declaretheorem[name=Lemma]{lem}

\declaretheorem[name=Corollary]{cor}

\theoremstyle{definition} \declaretheorem[name=Definition]{defin}

\theoremstyle{remark} \declaretheorem[name=Remark]{rem}

\newcommand\cA{{\mathcal A}} \newcommand\cB{{\mathcal B}}

\newcommand\cG{{\mathcal G}} \newcommand\cH{{\mathcal H}}
 \newcommand\cJ{{\mathcal J}}

\newcommand\cP{{\mathcal P}} 
\newcommand\cR{{\mathcal R}} \newcommand\cT{{\mathcal T}}
\newcommand\cS{{\mathcal S}} 
\newcommand\cU{{\mathcal U}} 

\newcommand\sL{{\mathscr L}}

 \newcommand\bN{{\mathbb N}}
 
 \newcommand\bR{{\mathbb R}}
 \newcommand\bZ{{\mathbb Z}}
 
\renewcommand\L{{\mathbf L}}
 \let\emptyset\varnothing
\newcommand\ve{\varepsilon} 
\newcommand\vs{\varsigma} 

\newcommand{\abs}[1]{|{#1}|} \newcommand\Id{{\mathds{1}}}

\newcommand\leb{{\mathbf m}} \newcommand\uspace{X}
\newcommand\metr{\mathbf{d}}

\newcommand{\expan}{{\Lambda}} 
 
\newcommand{\dist}{{D}} 
\newcommand{\distexp}{{\alpha}} 
 
\newcommand{\modreg}{{a_{0}}} 
\newcommand{\sigtail}{{\sigma}} 
\newcommand{\ntail}{{n_0}} \newcommand{\epstail}{\ve_0}
 \newcommand{\epstailtwo}{\ve_1}
 \newcommand{\epstailfour}{\ve_2}

\newcommand\Cball{C_{ball}}

\newcommand\Ca{{e^{\modreg\epstail^\distexp}}}
\newcommand\Caa{{e^{-\modreg\epstail^\distexp}}}

\newcommand\diam{\operatorname{diam}}
\newcommand\cl{\operatorname{cl}} \newcommand\dimn{d}
\newcommand\ball{\cB}

\newcommand\proper{B_{0}}

\numberwithin{equation}{section}

\author{Peyman Eslami}
\address{Peyman Eslami\\
  Dipartimento di Matematica\\
  II Universit\`{a} di Roma (Tor Vergata)\\
  Via della Ricerca Scientifica, 00133 Roma, Italy.} \email{{\tt
	peslami7@gmail.com}} \title[Inducing schemes]{Inducing schemes for
  multi-dimensional piecewise expanding maps} \keywords{}
  \date{ February 14, 2020.}
  \thanks{The author has been supported by the MIUR Excellence Department Project Grant awarded to the Department of Mathematics, University of Rome Tor Vergata, CUP E83C18000100006.}
\begin{document}
\begin{abstract} 
	We construct inducing schemes for general multi-dimensional piecewise expanding maps where the base transformation is Gibbs-Markov and the return times have exponential tails. Such structures are a crucial tool in proving statistical properties of dynamical systems with some hyperbolicity.
\end{abstract}
\maketitle

\section{Introduction}
\label{sec:intro}

Statistical properties of chaotic dynamical systems have been a subject of interest for mathematicians and physicists in the past several decades. While such properties are better understood for uniformly hyperbolic systems, the same cannot be said about systems with non-uniform hyperbolicity. The reason is that there are many mechanisms for non-uniform behaviour (e.g. intermittency, existence of critical points or singularities, etc.) and usually they are mixed with regions (or periods) of uniformly hyperbolic behaviour. 
To treat this difficulty Young \cite{You98, You99} proposed an abstract framework to study such systems. She showed that if the system admits a certain structure, called a Young tower, then statistical properties such as rates of decay of correlations can be deduced using the analogy to Markov chains. Since then many other statistical properties have been studied assuming the existence of such structures. However, constructing such structures for various systems is not easy and requires a good understanding of the nature of non-uniformity of hyperbolicity. Even then, it is usually done in a case by case basis. 

The purpose of this article is to obtain such structures for general multi-dimensional expanding systems with discontinuities, that is, when the nature of non-uniformity is the presence of discontinuities. 
This is the first time inducing schemes are constructed in an optimal
way (with exponential tails) for general multi-dimensional piecewise
expanding maps. As pointed out earlier, virtually every statistical property can then be derived from the existence of such structures (existence and properties of absolutely continuous invariant measures, decay of correlations, central limit theorem, large deviations,
Berry-Esseen theorem, almost sure invariance principle, law of
iterated logarithm, etc.). 

An example of a dynamical system to which this method applies, but previous
methods do not, is available in \cite{EMV19}. As shown in \cite{EMV19}, the  results of the current paper are also relevant to proving statistical properties for non-uniformly expanding dynamical systems. This is because through inducing, one can replace the mechanism of non-uniformity of hyperbolicity with the presence of discontinuities. All of this is worked out for a family of a multidimensional, nonMarkov, nonConformal intermittent maps in \cite{EMV19}. Other systems to which the current framework can be applied are ``Hu-Vaienti''-type maps \cite{HuVai09} and ``Viana''-type maps \cite{Via97, Al00}. Finally, another motivation for this work comes from the study of multi-dimensional dispersing billiards. We refer the reader to the survery \cite{Sz17} for understanding the relevance of this work in providing insights into the problems surrounding the study of multi-dimensional billiards. 

The paper is organized in a top-down format in the sense that the main theorems are stated and proved asumming the lemmas and the tools that appear later in the paper.

As the technical details of our proofs may obscure the ``big picture'' it may help to have the following anology in mind. Imagine you have a chunk of cookie dough and a few cookie cutters and you want to make cookies. Your main objective is to use \textit{all} of the dough for this purpose. So you may think that you just expand the dough, use a cookie cutter to cut out your cookies and whatever dough is left over you expand and cut in a similar way until you have finshed the whole dough. But life is not so easy because someone else (the piecewise expanding map a.k.a. the devil) cuts, expands and folds the dough in complicated ways (yet with some restrictions, namely hypotheses \ref{hyp1}--\ref{hyp3}). Our theorems essentially say that you can beat the devil in this game and the proofs are essentially recipes for making cookies at an efficient pace (in a multi-dimensional setting) regardless of what the devil does to complicate the task. A crucial step is to make sure that cookies are cut out in such a way that the left over dough can be reused and that eventually the whole dough is used. 

\section{Setting and assumptions}
\label{sec:setting} Consider $\bR^{\dimn}$ endowed with the Euclidean
metric $\metr$ and the Lebesgue measure $\leb$. Let $\uspace \in \bR^{\dimn}$ be a
bounded, Borel measurable subset such that
\begin{equation}
 \label{eq:bdofspace} \sup_{\ve>0} \frac{\leb{(\partial_{\ve}X)}}{\ve}
 <\infty,
\end{equation} where $\partial \uspace = \cl \uspace \cap \cl
 (\bR^{\dimn}\setminus \uspace)$ is the topological boundary of
 $\uspace$ in $\bR^{\dimn}$ and $\partial_{\ve}\uspace=\{x \in
 \uspace: \metr(x, \partial \uspace) < \ve\}$. We consider a
 \textit{non-singular piecewise invertible map} $T$ on $\uspace$ with
 respect to the countable partition $\cP = \{O_h\}_{h \in \cH}$ of
 open subsets of $\uspace$. This means that $\leb(\uspace \setminus
 \bigcup_{h \in \cH} O_{h})= 0$ and the restrictions $T:O_h \to
 T(O_h)$ and their inverses are non-singular (i.e. $\forall h \in \cH$,
 $(T|_{O_{h}})_{*}(\leb|_{O_{h}})$ is equivalent to $\leb|_{TO_{h}}$)
 homeomorphisms of $O_h$ onto $T(O_h)$. It is notationally convenient
 to use $h: TO_{h} \to O_{h}$ to also denote an inverse branch of $T$ and use $\cH$ to
 denote the set of inverse branches of $T$. Accordingly, we denote the
 set of inverse branches of $T^n$, $n \in \bN$, by $\cH^n$ and the
 corresponding partition by $\cP^{n}$. We write $Jh$ for the
 Radon-Nikodym derivative $d(\leb \circ h)/ d\leb$. We make the
 following assumptions on our dynamical system.

\begin{rem}
Note that $\uspace$ is contained in $\cl \uspace$ (closure of $\uspace$), which
is the disjoint union of $int\, \uspace$ (interior of $\uspace$) and
$\partial\uspace$. Moreover, $\leb(\partial \uspace)=0$ because of
\eqref{eq:bdofspace}. So, we can restrict our dyna{}mics to $int\,
\uspace$, which is an open set. So, without loss of generality, we may
assume that $\uspace$ is an open subset of $\bR^\dimn$. Consequently, 
$I \subset \uspace$ is
open in $\uspace$ if and only if it is open in $\bR^\dimn$.
 \end{rem}

\begin{hyp}[Uniform expansion]\label{hyp1} For every $h \in \cH$ and
$\ve >0$, denote
\begin{equation*}
 \expan_{h}(\ve)= \sup_{\{x,y \in T(O_{h})\,: \metr(x,y)\le \ve\}}
 \frac{\metr(h(x), h(y))}{\metr(x,y)}.
\end{equation*} There exist $\epstailtwo>0$ and $\expan \in (0,1)$
 such that for every $h \in \cH$, $ \expan_{h}(\epstailtwo) \leq
 \expan <1 $. Set $\expan_{h}:=\expan_{h}(\epstailtwo)$. Note that for
 $h \in \cH^{n}$, we can define $\expan_{h}$ using $T^{n}$ and it is
 easy to verify that for all $h \in \cH^{n}$, $\expan_{h}(\epstailtwo)
 \leq \expan^{n}<1$. \vspace{0.2 cm} \end{hyp}
\begin{hyp}[Bounded distortion]\label{hyp2} There exist $\alpha \in
(0,1]$, $\tilde\dist \geq 0$ such that $\forall h \in \cH$, $\forall
x,y \in T(O_h)$
\begin{equation}
\label{eq:dist} Jh(x) \leq e^{\tilde\dist \metr(x,y)^\distexp} Jh(y).
\end{equation} Let $\dist = \tilde \dist/(1-\expan^{\distexp})$. As a
 consequence of uniform expansion, \eqref{eq:dist} holds for $h \in
 \cH^{n}$ uniformly for all $n\in\bN$ with $D$ instead of $\tilde D$.
 \end{hyp} 
\begin{hyp}[Controlled complexity]\label{hyp3} There exist $\ntail \in
\bN$, $\epstailfour >0$ and $0\le\sigtail < \expan^{-\ntail}-1$ such
that for every open set $I$, $\diam I \leq \epstailfour$, for every
$\ve < \epstailfour$,
\begin{equation}
 \label{eq:dyncomplexity} \sum_{\{h \in \cH^{\ntail},\, \leb(I \cap
 O_{h})>0\}} \frac{\leb(h(\partial_{\ve}T^{\ntail}(I \cap
 O_{h}))\setminus
 \partial_{\expan^{\ntail}\ve}I)}{\leb(\partial_{\expan^{\ntail}
 \ve}I)} \leq \sigtail < \expan^{-\ntail}-1.
\end{equation} Moreover, there exists a constant $\bar C < \infty$
 such that for every integer $1 \leq r < \ntail$, for every $\ve <
 \epstailfour$,
\begin{equation}
 \label{eq:dyncomplexitybound} \sum_{\{h \in \cH^{r},\, \leb(I \cap
 O_{h})>0\}} \frac{\leb(h(\partial_{\ve}T^{r}(I\cap O_{h})) \setminus
 \partial_{\expan^{r} \ve}I)}{\leb(\partial_{\expan^{r} \ve}I)} \leq
 \bar C.
\end{equation} We refer to the expression on the left-hand side of
 \eqref{eq:dyncomplexity} as the \textit{complexity expression}.
 \end{hyp}
\begin{rem} Often one can check \eqref{eq:dyncomplexity} for $\ntail
=1$ in which case there is no need to check
\eqref{eq:dyncomplexitybound}. \end{rem}
\begin{rem} \label{bdaction} Suppose $h \in \cH$ and $T|_{O_{h}}:O_{h}
\to TO_{h}$ has an extension $\bar T_{h}: \cl O_{h} \to \cl TO_{h}$
that is invertible, its inverse $\bar h$ satisfies condition
\ref{hyp1} and $\partial(TO_{h}) \subset \bar T_{h}(\partial O_{h})$.
Then if $A \subset O_{h}$, we have $\forall \ve < \epstailtwo$,
\begin{equation*}
 \begin{split} h(\partial_{\ve}TA ) =\bar h(\partial_{\ve}TA) &= \bar
 h\{y\in T(A):\metr(y,\partial (TA) < \ve\} \\ &\subset \{x\in
 A:\metr(Tx, \bar T_{h}(\partial A)) < \ve\}\\ &\subset \{x \in A:
 \metr(x, \bar h(\partial(TA)))< \expan_{h}\ve\}\\ &\subset\{x \in A:
 \metr(x, \partial A)<\expan_{h}\ve\} = \partial_{\expan_{h}\ve}(A).
 \end{split}
\end{equation*} This is a simple but useful fact to keep in mind when
 checking \eqref{eq:dyncomplexity}. \end{rem}

Fix 
\begin{equation} \label{eq:regepsdef}
\modreg > \dist/(1-\expan^{\distexp})=\tilde
 \dist/(1-\expan^{\distexp})^{2},  \qquad \epstail = \min\{\epstailtwo, \epstailfour\}.
\end{equation}
  
\section{Statement of the main results}
\begin{thm}
\label{thm:GM} Suppose $T:\uspace \circlearrowleft$ satisfies
hypotheses \ref{hyp1}-\ref{hyp3}. There exist a refinement $\cP'$ of
the partition $\cP$ into open sets (mod $0$) and a function
$\tau:\uspace\to\bZ^+$ constant on elements of $\cP'$ such that

\begin{enumerate} [label=(\alph*)] \item The map $G=T^\tau:\uspace
\circlearrowleft$ is a Gibbs-Markov map with finitely many images.
\item $\leb(\tau>n)\le const \cdot \kappa^n$ for some $\kappa \in
(0,1)$.
\end{enumerate}
 \end{thm}
 
\begin{rem}
    By a Gibbs-Markov map we mean a piecewise expanding map having uniform expansion and bounded distortion.
\end{rem}
    Under an ergodicity assumption, the induced map can be
 upgraded to a full-branched Gibbs-Markov map.
\begin{cor} \label{cor:fullGM}Suppose $T:\uspace \circlearrowleft$
satisfies hypotheses \ref{hyp1}-\ref{hyp3} and $(T, \leb)$ is ergodic.
There exist an open set $Z\subset \uspace$ and a refinement $\cP''$ of the
partition $\cP$ into open sets (mod $0$) such that $Z$ is a union of
elements of $\cP''$ and there exists a map $\tilde \tau: Z \to\bZ^+$
constant on elements of $\cP''$ such that

\begin{enumerate} [label=(\alph*)] \item The map $\tilde G=
T^{\tilde\tau}:Z\circlearrowleft$ is a full-branched Gibbs-Markov map.
\item $\leb(\tilde \tau>n)\le const \cdot\tilde \kappa^n$ for some
$\tilde \kappa \in (0,1)$.
\end{enumerate}
\end{cor}

\begin{rem}
	It follows from the proof of \Cref{cor:fullGM} that     given
	any open set $U \subset \bR^\dimn$, there exists an open cube inside
	$U$ that can be chosen as the set $Z$.
\end{rem}

By \cite{You98, You99}, if $(T,\leb)$ is mixing and admits a
full-branched inducing scheme as above, then $\gcd\{n \ge 1:
\leb(\{\tilde \tau=n\})>0\} = 1$. The converse is also true, that is,
if we find a full-branched inducing scheme, as above, for which the
$\gcd$ of the return times is equal to $1$, then the system is mixing.
The next theorem provides such an inducing scheme under extra
conditions.

\begin{defin} We say that $Z \subset \uspace$ has a \emph{nice
	boundary} if there exists a constant $C_Z>0$ such that $ \forall
	\ve\ge 0$, $\leb(\partial_\ve Z) \le C_Z \ve$ and for every open
	set $I
	\subset\uspace$ containing $Z$,
	\begin{equation*}
		\leb(\partial_{\ve}(I \setminus \cl Z)\setminus\partial_\ve I)
		\le C_Z\leb(\partial_\ve I).
	\end{equation*}
\end{defin}
\begin{rem} Many geometric shapes have nice boundaries. For example,
	sets with piecewise smooth boundaries (no cusps) including rectangles and
	balls.
\end{rem}

\begin{defin} We say that $Z \subset \uspace$ is \emph{fully
recurrent} (at times $\{n_j\}_{j=1}^K$) if there exist $K \ge 2$
positive integers $\{n_j\}_{j=1}^K$ such that $\gcd\{n_j\}_{j=1}^K=1$
and the following holds: There exist inverse branches $h_{n_1} \in
\cH^{n_1}$, $h_{n_2}\in \cH^{n_2}$, \dots, $h_{n_K} \in \cH^{n_K}$
such that $O_{h_{n_1}}, \dots, O_{h_{n_K}}$ are pairwise disjoint and
for every $j = 1, \dots, K$, $$T^{n_j}(O_{h_{n_j}} \cap Z) \supset
Z.$$
\end{defin}

\begin{thm}
\label{thm:fullbranch} Suppose $T:\uspace\circlearrowleft$ satisfies
	hypotheses \ref{hyp1}-\ref{hyp3}. In addition, suppose that for
	every $\delta >0$ there exists $Z\subset \uspace$ of $\diam Z \le
	\delta$ that is fully recurrent and has a nice boundary. Then,
	there exist $\delta'$ such that $\forall \delta \le \delta'$ there exist a refinement $\cP''$ of the partition $\cP$ into open
	sets (mod $0$) such that $Z=Z(\delta)$ is a union of elements of $\cP''$ and
	there exists a map $\tilde\tau: Z \to\bZ^+$ constant on elements
	of $\cP''$ such that
	\begin{enumerate}[label=(\alph*)]
  		\item \label{prop-fullM-fullbr} The map $\tilde
  		G=T^{\tilde\tau}:Z\circlearrowleft$ is a full-branched
  		Gibbs-Markov map.
		\item \label{prop-fullM-gcd} $\gcd\{n \ge 1: \leb(\{\tilde
		\tau=n\})>0\} = 1$.
		\item \label{prop-fullM-tails}$\leb(\tilde \tau>n)\le
		const\cdot\tilde \kappa^n$for some $\tilde \kappa \in (0,1)$.
	\end{enumerate}
\end{thm}

\begin{thm}
	\label{thm:fullbranchgcd-one} Suppose $T:\uspace\circlearrowleft$ satisfies
		hypotheses \ref{hyp1}-\ref{hyp3}. In addition, suppose that for
		every $\delta >0$ there exist $Z \subset Z' \subset \uspace$ such that $Z$ has a nice boundary, $\diam Z' \le
		\delta$, $\leb(Z') > \leb(Z)$ and there exists $h \in \cH$ such that $O_h \subset Z$ and $TO_h \supset Z'$. Then,
		there exist $\delta'$ such that $\forall \delta \le \delta'$ there exist a refinement $\cP''$ of the partition $\cP$ into open
		sets (mod $0$) such that $Z=Z(\delta)$ is a union of elements of $\cP''$ and
		there exists a map $\tilde\tau: Z \to\bZ^+$ constant on elements
		of $\cP''$ such that
		\begin{enumerate}[label=(\alph*)]
			  \item \label{prop-fullM-fullbr-one} The map $\tilde
			  G=T^{\tilde\tau}:Z\circlearrowleft$ is a full-branched
			  Gibbs-Markov map.
			\item \label{prop-fullM-gcd-one} $\leb(O_h \cap \{\tilde
			\tau=1\}) > 0$.
			\item \label{prop-fullM-tails-one}$\leb(\tilde \tau>n)\le
			const\cdot\tilde \kappa^n$for some $\tilde \kappa \in (0,1)$.
		\end{enumerate}
	\end{thm}
 
\section{Proof of \Cref{thm:GM} and \Cref{cor:fullGM}}
\label{sec:proof1}
Note that $\modreg, \epstail$ and $\proper, \delta_0$ are constants
that depend on the map and are fixed once and for all once $T$ is fixed. $\modreg,
\epstail$ are defined in \eqref{eq:regepsdef} and $\proper, \delta_0$ are
defined in \Cref{sec:toolbox} in \Cref{prop:bd_invariance} and \Cref{rem:delzero}.
\begin{proof}[Proof of \Cref{thm:GM}] The following steps lead to our
sought after inducing scheme.

\begin{enumerate} [label=(\arabic*), leftmargin=*] \item Fix
$\delta=\delta_0$ and consider the partition $\cR$ of $\uspace$ given
by \Cref{lem:partition1}. Let us focus on defining the inducing scheme
on one element of this partition. The same can be done for all other
partition elements and in a uniform way because $\cR$ is finite. Fix
$R \in \cR$ and let $\cG_{0}=\{(R,
\Id_{R}/\leb(R))\}$ and $w_{0}=\leb(R)>0$. Due to \cref{Ritemone} of
\Cref{lem:partition1}, the singleton family $\cG_{0}$ with associated
weight $\{w_{0}\}$ is an $(\modreg, \epstail, B)$--proper standard
family (\Cref{std_family}) for some constant $B >0$ possibly larger than $\proper$. \item \label{recstep} By
\Cref{prop:bd_invariance}, $\cG_{1}:=\cT^{n_{rec}(B)} \cG_{0}$ is an
$(\modreg, \epstail,
\proper)$--proper standard family. \item \label{repeatone} By
\cref{Ritemtwothree} of \Cref{lem:partition1}, every standard pair in
$\cG_{1}$ whose domain is $\delta_{0}$-regular contains at least one
element $R'$ from the collection $\cR$. ``Stop'' each
$\delta_0$-regular standard pair of $\cG_{1}$ on its corresponding
rectangle $R'\in \cR$. By stopping we mean going back to $R$ and
defining the return time $\tau = n_{rec}(B)$ on the subset of $R$ that
maps onto $R'$ under $\cT^{n_{rec}(B)}$. By \Cref{lem:fixedratio},
applied to $\cG_1$, the ratio of the removed weight from $\cG_{1}$ to
the weight of the remainder family (defined in
\Cref{lem:remainderfam}), which we denote by $\hat \cG_{1}$, is at
least some positive constant $t$ given by \eqref{eq:t}. Since the
weight of standard pairs is preserved under iteration, this
corresponds to defining $\tau$ on a subset $A \subset R$ such that
$\leb(A) \ge t \cdot \leb(R \setminus A)$. Also, by
\Cref{lem:remainderfam}, $\hat \cG_{1}$ is an $(\modreg, \epstail,
\bar C_{\cR} \proper)$--standard family. \item \label{repeattwo} Just
as in step \ref{recstep}, $\cG_{2}:=\cT^{n_{rec}(\bar
C_{\cR}\proper)}\hat \cG_{1}$ is an $(\modreg, \epstail,
\proper)$--proper standard family so we can apply step \ref{repeatone}
to it. \item Repeat the steps \ref{repeatone}, \ref{repeattwo}
$\rightarrow$ \ref{repeatone}, \ref{repeattwo} $\rightarrow \cdots$,
incrementing the indices accordingly during the process.
\end{enumerate} Applying the above inductive procedure, we will get a
 ``stopping time'' (or return time) $\tau: \uspace \to \bN$ defined on
 a (mod $0$)-partition $\cP'$ of $\uspace$. $\cP'$ is a refinement of
 the partition $\cP$ and $\tau$ is constant on each element of $\cP'$.
 The return time $\tau$ will have exponential tails because at each
 step (where the time between steps is universally bounded by
 $n_{rec}(B)+n_{rec}(\bar C_{\cR}\proper)$) it is defined on a set $A
 \subset X$, where $\leb(A) \ge t \cdot \leb(X\setminus A)$. By
 construction the induced map has finitely many images which form a
 sub-collection of $\cR$. Note that distortion bound is always
 maintained under iterations of $T$ by assumptions \ref{hyp1} and
 \ref{hyp2} so we need not worry about it.
\end{proof}
 
\begin{proof}[Proof of \Cref{cor:fullGM}] Let $Z$ be one of the
finitely many images of $G$. Suppose $Z$ is minimal in the sense that
no proper subsets of $Z$ is an image of $G$. Let $\vs:Z \to \bN$ be
the first return time of $G$ to $Z$ and $\tilde G = G^{\vs}: Z
\circlearrowleft$ be the associated first return map. Let $\cP''$ be
the partition of $\tilde G$, which is a refinement of $\cP'$ and hence
of $\cP$. Since $(T, \leb)$ is ergodic, so is $(G, \leb)$ hence $\vs$
is well defined. Since $G$ is Markov and $Z$ is minimal, $\tilde G$ is
full-branched. Define $\tilde \tau = \sum_{\ell=0}^{\vs-1} \tau \circ
G^{\ell} : Z \to \bN$, then $\tilde G = T^{\tilde \tau}$. Since $G$ is
a Markov map with finitely many images, $\vs$ has exponential tails
and therefore $\tilde \tau:\uspace \to \bN$ also has exponential
tails.
\end{proof}
  
\section{Proof of \Cref{thm:fullbranch}}
\label{sec:proof2}
\begin{proof}[Proof of \Cref{thm:fullbranch}]
  We follow a line of reasoning similar to that of the proof of
	\Cref{thm:GM}, but with some modifications when dealing with $R=Z$
	mainly in order to achieve \cref{prop-fullM-gcd} of
	\Cref{thm:fullbranch}.

	\begin{enumerate}[label=(\arabic*), leftmargin=*]
	\item Fix $\delta = \delta_0$ and $c=1/(2^{\dimn+2}V^{\dimn}\sqrt{\dimn})$.  Let $Z=Z(c\delta)$ be as in the hypothesis of \Cref{thm:fullbranch}. Note that $Z$ is also contained in a $\dimn$-dimensional cube of side length $c\delta$. Let $\cR=\cR(\delta)$ be the partition given by \Cref{lem:partition1}. Let $\cG_{0}= \{(Z, \Id_{Z}/\leb(Z))\}$ and $w_{0}=\leb(Z)$.
		$\cG_{0}$ is an $(\modreg, \epstail, B)$--proper standard family
		for some $B>0$. Applying \Cref{lem:fully_rec} with $C_1=n_{rec}(B)$ and
		$C_2=n_{rec}(\bar C_{\cR}\proper)$, it follows that  $Z$ is fully
		recurrent at times $\{n_j\}_{j=1}^K$, where
		\begin{equation*} n_{1}\ge n_{rec}(B) \text{ and } n_{j+1}-n_{j}
			\ge n_{rec}(\bar C_{\cR}\proper),\ \forall j \in \{1,\dots,
			K-1\}.
		\end{equation*}
	\item Let $\cG_{1}:=\cT^{n_{1}}\cG_{0}$, taking $V_{*}= \cl Z$ as
		the set to avoid under $\cT^{n_{1}}$ under artificial chopping.
		This can be done due to \Cref{lem:divisibility}. Since $n_{1} \ge
		n_{rec}(B)$, $\cG_{1}$ is an $(\modreg, \epstail,
		\proper)$--proper standard family.
		\begin{enumerate}[label=(\arabic{enumi}.\arabic*), leftmargin=*]
		\item There exists a standard pair in $\cG_1$ whose domain
				contains $Z$. ``Stop'' it on $Z$. That is, define
				$\tau = n_{1}$ on
				\begin{equation*} A_{1}:=h_{n_1}(Z)\cap Z = \{x \in
					O_{h_{n_1}} \cap Z : T^{n_1}x \in Z\}.
				\end{equation*}  Note that $T^{n_1}A_1 = Z$. By
				\Cref{lem:remainderfam}, the remainder from $\cG_{1}$,
				which we denote by $\hat
				\cG_{1}$ is an $(\modreg, \epstail, \bar
				C_{\cR}\proper)$--proper standard family. Let
				$\cG_{2}=\cT^{n_{2}-n_{1}}\hat\cG_{1}$. Since
				$n_{2}-n_{1} \ge n_{rec}(\bar C_{\cR}\proper)$,
				$\cG_{2}$ is an $(\modreg,
				\epstail, \proper)$--proper standard family.
		\item As before, define $\tau = n_{2}$ on $A_2 := h_{n_2}(Z) \cap Z$. Note that $T^{n_2}A_2 = Z$ and $A_2$ is disjoint from $A_1$ because $O_{h_{n_2}}$ is disjoint from
				$O_{h_{n_1}}$. By
				\Cref{lem:remainderfam}, the remainder from
				$\cG_{2}$, which we denote by $\hat \cG_{2}$ is an
				$(\modreg, \epstail, \bar C_{\cR}\proper)$--proper
				standard family. Let
				$\cG_{3}=\cT^{n_{3}-n_{2}}\hat\cG_{2}$. Since
				$n_{3}-n_{2} \ge n_{rec}(\bar C_{\cR}\proper)$,
				$\cG_{3}$ is an $(\modreg, \epstail, \proper)$--proper
				standard family.
		\item We continue this process until we define $\tau = n_{K}$
			on $$A_{K}:=h_{n_K}(Z) \cap Z,$$ which is
			disjoint from previous $A_j$'s.

			Let $\hat \cG_{K}$ be the remainder from $\cG_{K}$. Note
			that $\hat \cG_{K}$ is an $(\modreg, \epstail,
			\bar C_{\cR}\proper)$--proper standard family. Also note
			that $\forall j \in \{1,\dots, K\}$, $A_{j} \subset Z$ and
			$T^{n_{j}}A_{j}=Z$. Moreover, $\leb(A_{j})>0$ because
			$\forall j
			\in \{1,\dots, K\}$ the inverse branches of $T^{n_{j}}$
			are non-singular, there are at most countably many such
			branches and $\leb(Z)>0$.
		\end{enumerate}
	\item We have achieved that
		\begin{equation*}
			\gcd\left\{n: \leb\left(\{\tau=n\} \cap
					\bigcup_{j=1}^{K}A_{j}\right) > 0\right\} =1.
		\end{equation*} Also, by construction, $T^\tau$ maps each
		$A_j$ onto $Z$ in a one-to-one fashion.

		We continue the construction of $\tau$ on the rest of $Z$,
		i.e. on $\hat Z := Z \setminus \bigcup_{j=1}^{K} A_{j}$, in
		such a way that it has exponential tails. We will do so by
		continuing to iterate $\hat \cG_{K}$.
	\item \label{Zrepeatone} Let $\cG_{K+1} = \cT^{n_{rec}(\bar
		C_{\cR}\proper)}\hat
		\cG_{K}$. Then $\cG_{K+1}$ is an $(\modreg, \epstail,
		\proper)$--proper standard family. By \cref{Ritemtwothree} of
		\Cref{lem:partition1}, every standard pair in $\cG_{K+1}$
		whose domain is $\delta_{0}$--regular contains an element
		$R_{k}$, $1
		\le k \le N$, from the collection $\cR$. ``Stop'' such
		standard pairs of $\cG_{K+1}$ on $R_{k} \in \cR$. By stopping
		we mean going back to $\hat Z \subset Z$ and defining the
		return time $\tau = n_{K}+n_{rec}(\bar C_{\cR}\proper)$ on the
		subset of $\hat Z$ that maps onto $R_{k}$ under
		$T^{n_{K}+n_{rec}(\bar C_{\cR}\proper)}$. By
		\Cref{lem:fixedratio}, the ratio of the removed weight from
		$\cG_{K+1}$ to the weight of the remainder family, which we
		denote by $\hat \cG_{K+1}$, is at least some positive constant
		$t$ given by \eqref{eq:t}. Note that since the total weight is
		preserved under iteration, this corresponds to defining $\tau$
		on a subset $A \subset \hat Z$ such that $\leb(A) \ge t \cdot
		\leb(\hat Z
		\setminus A)$. Also, by \Cref{lem:remainderfam},
		$\hat\cG_{K+1}$ is an $(\modreg, \epstail, \bar C_{\cR}
		\proper)$--standard family.

	\item \label{Zrepeattwo} $\cG_{K+2}:=\cT^{n_{rec}(\bar
		C_{\cR}\proper)}\hat \cG_{K+1}$ is an $(\modreg, \epstail,
		\proper)$--proper standard family so we can apply step
		\ref{Zrepeatone} to it.
	\item Repeat the steps \ref{Zrepeatone}, \ref{Zrepeattwo}
		$\rightarrow$ \ref{Zrepeatone}, \ref{Zrepeattwo} $\rightarrow
		\cdots$, incrementing the indices accordingly during the
		process. This procedure defines $\tau$ on $\hat Z$ up to a
		measure zero set of points (which includes points that map
		into $\partial Z$).
	\end{enumerate}

	The above steps described how to define $\tau$ on $Z$. We have
	also explained how to define $\tau$ on the rest of the elements of
	$\cR$ in the proof of \Cref{thm:GM}. Putting these together we get
	the same statement as \Cref{thm:GM}, but with the additional
	properties that $\gcd\{n:
	\leb(\{\tau=n\} > 0\} =1$, $Z$ is one of the finitely many images
	of $G=T^{\tau}$ and that $G(Z) \supset Z$.

	Let $\vs:Z \to \bN$ be the first return time of $G$ to $Z$ and
	$\tilde G = G^{\vs}: Z \circlearrowleft$ be the associated first
	return map. Since $GA_{j} =T^{\tau}A_{j}=T^{n_{j}}A_{j}=Z$,
	$\forall j \in \{1, \dots, K\}$, it follows that $\vs = 1$ on the
	set $\bigcup_{j=1}^{K} A_{j}$.

	Define $\tilde \tau = \sum_{\ell=0}^{\vs-1} \tau \circ G^{\ell} :
	Z
	\to
	\bN$, then $\tilde G = T^{\tilde \tau}$. It follows from the
	previous paragraph that $\tilde \tau = \tau$ on $\bigcup_{j=1}^{K}
	A_{j} \subset Z$. This implies \cref{prop-fullM-gcd}.
	\Cref{prop-fullM-fullbr} and \cref{prop-fullM-tails} simply follow
	from the fact that $G$ is a Markov map with finitely many states
	(hence $\vs$ has exponential tails) and $\tau:\uspace \to \bN$ has
	exponential tails.
\end{proof}

\section{Proof of \Cref{thm:fullbranchgcd-one}}
The proof of \Cref{thm:fullbranchgcd-one} proceeds similarly to the proof of \Cref{thm:fullbranch} except that in the initial step we need to define the stopping time pre-maturely because the initial family is not an $(\modreg, \epstail, \proper)$--standard family and we cannot iterate to make it so. The remedy is to use \Cref{lem:remainderfam-one} where we had previously used \Cref{lem:remainderfam}.

\begin{proof}
\begin{enumerate}[label=(\arabic*), leftmargin=*]
	\item Fix $\delta = \delta_0$ and $c=1/(2^{\dimn+2}V^{\dimn}\sqrt{\dimn})$.  Let $Z=Z(c\delta)$ be as in the hypothesis of \Cref{thm:fullbranchgcd-one}. Note that $Z$ is also contained in a $\dimn$-dimensional cube of side length $c\delta$. Let $\cR=\cR(\delta)$ be the partition given by \Cref{lem:partition1}. Let $\cG_{0}= \{(Z, \Id_{Z}/\leb(Z))\}$ and $w_{0}=\leb(Z)$.
		$\cG_{0}$ is an $(\modreg, \epstail, B)$--proper standard family
		for some $B>0$. 
	\item Let $\cG_{1}:=\cT\cG_{0}$, taking $V_{*}= \cl Z'$ as
		the set to avoid under $\cT$ under artificial chopping.
		This can be done due to \Cref{lem:divisibility}. Let $h$ be as in the hypotheses of \Cref{thm:fullbranchgcd-one}. Define $\tau = 1$ on
		\begin{equation*} A_{1}:=h(Z)\cap Z = \{x \in
			O_{h} \cap Z : Tx \in Z\}.
		\end{equation*}  Note that $TA_1 = Z$. By
		\Cref{lem:remainderfam-one}, the remainder from $\cG_{1}$,
		which we denote by $\hat \cG_{1}$ is an $(\modreg, \epstail, B')$--proper standard family for some $B'>0$. Let $\hat Z := Z \setminus A_{1}$
	\item \label{recstep-one} let
		$\cG_{2}:=\cT^{n_{rec}(B')} \cG_{0}$. Then $\cG_2$ is an
		$(\modreg, \epstail,
		\proper)$--proper standard family. 
	\item \label{repeatone-one} By
		\cref{Ritemtwothree} of \Cref{lem:partition1}, every standard pair in
		$\cG_{2}$ whose domain is $\delta_{0}$-regular contains at least one
		element $R'$ from the collection $\cR$. ``Stop'' each
		$\delta_0$-regular standard pair of $\cG_{2}$ on its corresponding
		rectangle $R'\in \cR$. By stopping we mean going back to $Z$ and
		defining the return time $\tau = 1 + n_{rec}(B')$ on the subset of $R$ that
		maps onto $R'$ under $\cT^{1+n_{rec}(B')}$. By \Cref{lem:fixedratio},
		applied to $\cG_2$, the ratio of the removed weight from $\cG_{2}$ to
		the weight of the remainder family (defined in
		\Cref{lem:remainderfam}), which we denote by $\hat \cG_{2}$, is at
		least some positive constant $t$ given by \eqref{eq:t}. Since the
		weight of standard pairs is preserved under iteration, this
		corresponds to defining $\tau$ on a subset $A \subset R$ such that
		$\leb(A) \ge t \cdot \leb(R \setminus A)$. Also, by
		\Cref{lem:remainderfam}, $\hat \cG_{2}$ is an $(\modreg, \epstail,
		\bar C_{\cR} \proper)$--standard family. 
	\item \label{repeattwo-one} Just
		as in step \ref{recstep-one}, $\cG_{3}:=\cT^{n_{rec}(\bar
		C_{\cR}\proper)}\hat \cG_{2}$ is an $(\modreg, \epstail,
		\proper)$--proper standard family so we can apply step \ref{repeatone-one}
		to it. 
		\item Repeat the steps \ref{repeatone-one}, \ref{repeattwo-one}
		$\rightarrow$ \ref{repeatone-one}, \ref{repeattwo-one} $\rightarrow \cdots$,
		incrementing the indices accordingly during the process.
	\end{enumerate}

	The above steps described how to define $\tau$ on $Z$ so that $\leb(O_h \cap \{\tilde
	\tau=1\}) > 0$. We have
	also explained how to define $\tau$ on the rest of the elements of
	$\cR$ in the proof of \Cref{thm:GM}. The rest of the proof is the same as the proof of \Cref{thm:fullbranch}.
\end{proof}
 
\section{Supplementary lemmas}
\label{sec:inducing}
This section contains supplementary lemmas for the proofs of our main theorems. The first lemma is taken from \cite{BT08} and stated in a form
that is suitable for our needs.
\begin{lem}[Sublemma C.1 of \cite{BT08}] \label{BT} Suppose $I$ is a
non-empty measurable bounded subset of $\bR^{\dimn}$ and $E$ is a
hyperplane cutting $I$ into left and right parts $I_{l}$ and $I_{r}$.
Then $\forall \ve \ge 0$ and $0 \le \xi \le 1$, we have
\begin{equation}
 \label{eq:BT}
\begin{split}
\leb(\{x \in I_{l}: \metr(x, E)\leq \xi \ve\}\setminus
\{x \in I: \metr(x, \partial I) \leq \ve\}) &\leq \\ & \hspace{-3cm}
\xi \leb(\{x \in I_{r}: \metr(x, \partial I)\le \ve\}).
\end{split}
\end{equation}
\end{lem}
 
\begin{lem}
\label{lem:partition1} There exists a constant $c>0$ such that for every $\delta >0$ and every open $Z \subset X$ that has a nice boundary and is contained in a $d$-dimensional cube of side-length $c\delta$, there
exist a finite (mod $0$)-partition $\cR=\{R_{j}\}_{j=1}^{N}$ of
$\uspace$ into open sets such that
\begin{enumerate} [label=(\arabic*)] 
\item $Z \in \cR$.
\item \label{Ritemone} for every
$1 \le j \le N$, $\sup_{\ve>0} \ve^{-1}\leb(\partial_{\ve}R_{j})
<\infty$, \item \label{Ritemtwothree} for every $\delta$--regular set
$I$, there exists $R \in \cR$ s.t. $I \supset
R$ and
\begin{eqnarray} \label{eq:621} \leb(I \setminus R)&\ge&
(1/2)\leb(I);\\ \label{eq:622}\leb(\partial_{\ve}(I \setminus \cl
R)\setminus\partial_{\ve}I) &\le& \max\{2\dimn, C_Z\} \leb(\partial_{\ve}I).
\end{eqnarray}
\end{enumerate}
 \end{lem}
 
\begin{proof} Let $c=1/(2^{\dimn+2}V^{\dimn}\sqrt{\dimn})$, where
$V=\leb(\ball_{1})$ is the volume of the unit ball in $\bR^{\dimn}$,
and let $\cS=\{S_{j}\}$ denote a grid of open cubes in $\bR^{\dimn}$
with sides of length $c\delta$ parallel to the coordinate axes. Since
$\uspace$ is bounded, the collection $\cR=\{Z\} \cup \{S \cap (\uspace\setminus Z): S \in
\cS\}$ forms a finite (mod $0$) partition of $\uspace$ into open sets.
Now, if $R=Z$, then \cref{Ritemone} is satisfied because $Z$ has nice boundary. If $R = S \cap (X\setminus Z)$, then
\begin{equation*}
	\begin{split}
		\leb(\partial_\ve R) &\le \leb(\partial_\ve S) + \leb(\partial_\ve(\uspace \setminus Z)) \\
		&\le \leb(\partial_\ve S) + \leb(\partial_\ve(\uspace \setminus Z)\setminus \partial_\ve \uspace) + \leb(\partial_\ve \uspace). 
	\end{split}
\end{equation*}
By a simple calculation $\le 2\dimn\ve (c\delta)^{\dimn-1}$. Now
\cref{Ritemone} follows because $Z$ has a nice boundary.

Next, suppose $I \subset \uspace$ is a $\delta$-regular set. Then it
contains a ball $\ball(x, \delta)$ of radius $\delta$. Let $R \in \cR$
be the element that contains $x$. Since $\diam R  \le
\sqrt{\dimn}c\delta$,
\begin{equation*} R \subset \ball(x, \sqrt{\dimn}c\delta ) \subset
 \ball(x, \frac12\delta) \subset I.
\end{equation*} Moreover,
\begin{equation*}
 \leb(R)\le(c\delta)^{\dimn}=
 \frac{2^{\dimn}c^{\dimn}}{V^{\dimn}}\leb(\ball(x, \frac{1}{2}\delta))
 <\frac12 \ball(x, \frac12\delta) \le \frac{1}{2}\leb(I).
\end{equation*} It follows that $\leb(I \setminus R)\ge \frac12
 \leb(I)$ verifying \eqref{eq:621}. Note that either $R=Z$ or $R=S$
 for some $S \in \cS$. In the former case \eqref{eq:622} holds because
 $Z$ has a nice boundary. In the case that $R=S$, each of the $2\dimn$
 sides of the cube $R$ can be continued as a hyperplane to cross $I$.
 By \Cref{BT}, the $\ve$-boundary of each side contributes no more
 than the $\ve$-boundary of $I$, verifying \eqref{eq:622}.
\end{proof} Set $c_{\cR}=1/2$ and $C_{\cR}=2\max\{2\dimn, C_Z\}$.

\begin{lem} [Remainder family $\hat \cG$] \label{lem:remainderfam}
Suppose $\cG$ is an $(\modreg, \epstail, \proper)$--proper standard
family. Let $\hat \cG$ be the family obtained from $\cG$ by replacing
each $(I, \rho)$ of weight $w$ having a $\delta_0$-regular domain and containing an element $R=R(I) \in
\cR$ in its domain with $\leb(I\setminus R)\neq 0$, by $(I\setminus
\cl R, \rho \Id_{I
\setminus \cl R}/\int_{I \setminus R} \rho)$ of weight $w \int_{I
\setminus R} \rho$. Then $\hat \cG$ is an $(\modreg, \epstail, \bar
C_{\cR}\proper )$--proper standard family, where $\bar C_{\cR} = (\Ca
C_{\cR}+1) \Ca c_{\cR}^{-1}$. \end{lem}
 
\begin{proof} This is a consequence of \cref{Ritemtwothree} of
\Cref{lem:partition1}. Indeed, assuming $\cG=\{(I_{j}, \rho_{j})\}$
with associated weights $w_{j}$, we have, $\forall \ve < \epstail$,
\begin{equation*}
\begin{split} \abs{\partial_{\ve}\hat \cG} &\le \sum_{j} w_{j}
\int_{\partial_{\ve}(I_{j}\setminus \cl R)}\rho_{j} \le \sum_{j} w_{j}
\left( \int_{\partial_{\ve}(I_{j}\setminus \cl R)\setminus
\partial_{\ve}I_{j}}\rho_{j}
+\int_{\partial_{\ve}I_{j}}\rho_{j}\right) \\ &\le \sum_{j}
w_{j}\left( \Ca \frac{\leb(\partial_{\ve}(I_{j}\setminus \cl
R)\setminus \partial_{\ve}I_{j})}{\leb(\partial_{\ve}I_{j})}
\int_{\partial_{\ve}I_{j}}\rho_{j}
+\int_{\partial_{\ve}I_{j}}\rho_{j}\right)\\ &\le (\Ca
C_{\cR}+1)\abs{\partial_{\ve}\cG}, \end{split}
\end{equation*} where in the second line we have used the
 Comparability \Cref{Fed} and in the last line we have used
 \eqref{eq:622}. Since $\cG$ is $\proper$--proper,
 $\abs{\partial_{\ve}\cG} \le \proper \ve \abs{\cG}$; moreover
 \eqref{eq:621} can be used to show that $\abs{\cG} \le \Ca
 c_{\cR}^{-1} \abs{\hat\cG}$. Indeed, by \Cref{Fed},
\begin{equation*}
 \abs{\hat\cG} \ge \sum_{j}w_{j}\int_{I_{j}\setminus R} \rho_{j} \ge
 \sum_{j} w_{j} \Caa \frac{\leb(I_{j}\setminus
 R)}{\leb(I_{j})}\int_{I_{j}}\rho_{j} \ge \Caa c_{\cR}\abs{\cG}.
\end{equation*} It follows that $\abs{\partial_{\ve}\hat \cG} \le \bar
 C_{\cR} \proper \ve \abs{\hat \cG}$.
\end{proof}

\begin{lem} [Remainder family $\hat \cG$ in the presence of $Z$] \label{lem:remainderfam-one}
	Suppose $\cG$ is an $(\modreg, \epstail, B)$--proper standard
	family. Suppose $Z \subset Z' \subset \uspace$, $Z$ has a nice boundary and $\leb(Z') > \leb(Z)$. Let $\hat \cG$ be the family obtained from $\cG$ by replacing
	each $(I, \rho)$ of weight $w$ containing $Z'$ in its domain, by $(I\setminus
	\cl Z, \rho \Id_{I
	\setminus \cl Z}/\int_{I \setminus Z} \rho)$ of weight $w \int_{I
	\setminus Z} \rho$. Then $\hat \cG$ is an $(\modreg, \epstail,B' )$--proper standard family, for some constat $B'>0$. \end{lem}
	 
	\begin{proof} Assuming $\cG=\{(I_{j}, \rho_{j})\}$
	with associated weights $w_{j}$, we have, $\forall \ve < \epstail$,
	\begin{equation*}
	\begin{split} \abs{\partial_{\ve}\hat \cG} &\le \sum_{j} w_{j}
	\int_{\partial_{\ve}(I_{j}\setminus \cl Z)}\rho_{j} \le \sum_{j} w_{j}
	\left( \int_{\partial_{\ve}(I_{j}\setminus \cl Z)\setminus
	\partial_{\ve}I_{j}}\rho_{j}
	+\int_{\partial_{\ve}I_{j}}\rho_{j}\right) \\ &\le \sum_{j}
	w_{j}\left( \Ca \frac{\leb(\partial_{\ve}(I_{j}\setminus \cl
	Z)\setminus \partial_{\ve}I_{j})}{\leb(\partial_{\ve}I_{j})}
	\int_{\partial_{\ve}I_{j}}\rho_{j}
	+\int_{\partial_{\ve}I_{j}}\rho_{j}\right)\\ &\le (\Ca
	C_{\cR}+1)\abs{\partial_{\ve}\cG}, \end{split}
	\end{equation*} where in the second line we have used the
	 Comparability \Cref{Fed}. In the last line we have used a modified version of
	 \eqref{eq:622}. Note that $I_j$ is not necessarily $\delta_0$-regular, but since $I_j \supset Z$ and $Z$ has a nice boundary the same arguments in the proof of \Cref{lem:partition1} imply that $\leb(\partial_{\ve}(I \setminus \cl
	 Z)\setminus\partial_{\ve}I) \le C_Z \leb(\partial_{\ve}I)$.

	 Since $\cG$ is $B$--proper,
	 $\abs{\partial_{\ve}\cG} \le B\ve \abs{\cG}$. By \Cref{Fed}, and since $\leb(I_j \setminus Z) \ge \leb(Z'\setminus Z) > 0$, 
	\begin{equation*}
	 \abs{\hat\cG} \ge \sum_{j}w_{j}\int_{I_{j}\setminus Z} \rho_{j} \ge
	 \sum_{j} w_{j} \Caa \frac{\leb(I_{j}\setminus
	 Z)}{\leb(I_{j})}\int_{I_{j}}\rho_{j} \ge \emph{const.} \abs{\cG}.
	\end{equation*} It follows that $\abs{\partial_{\ve}\hat \cG} \le
	 B'\ve \abs{\hat \cG}$ for some constant $B'>0$.
	\end{proof}
 
\begin{lem}
 \label{lem:fixedratio} Let $\cR=\{R_{k}\}_{k=1}^{N}$ be the partition
 from \Cref{lem:partition1}. There exists a constant $t >0$ such that
 if $\cG=\{(I_{j}, \rho_{j})\}_{j \in \cJ}$ is an $(\modreg, \epstail,
 \proper)$--proper standard family, then
\begin{equation}
 \label{eq:bound2} \sum_{j \in \cJ_{reg}} w_{j} \int_{R(I_{j})}
 \rho_{j} \geq t \cdot \left(\sum_{j \notin \cJ_{reg}} w_{j} + \sum_{j
 \in \cJ_{reg}} w_{j} \int_{I_{j} \setminus R(I_{j})} \rho_{j}
 \right),
\end{equation} where $\cJ_{reg}$ is the set of $j \in \cJ$ such that
 $I_{j}$ is $\delta_{0}$--regular \end{lem}
 
\begin{proof} Since $\cG$ is an $(\modreg, \epstail, \proper)$--proper
standard family, at least $2/3$ of its weight is concentrated on
$(a_{0}, \epstail)$--standard pairs $(I, \rho)$, where $I$ is a
$\delta_{0}$--regular set (recall that $\delta_0 = 1/(3\proper)$). By
\cref{Ritemtwothree} of \Cref{lem:partition1}, each such standard pair
contains an element from the collection~$\cR$. Using this fact and the
regularity of standard pairs (recall (\ref{eq:comp1})), the left-hand
side of \eqref{eq:bound2} is
\begin{equation*}
 \geq (2/3)\abs{\cG}\Caa \leb(R(I_{j}))/\leb(I_{j}) \geq (2/3)\Caa
 \Cball(\epstail)^{-1}\leb(R(I_{j})),
\end{equation*}
where $\Cball(\epstail)$ denotes the measure of a ball of radius $\epstail$.
 Now consider the expression in the parentheses and on
 the right-hand side of \eqref{eq:bound2}. The first term of this
 expression is the total weight of the standard pairs that are not
 $\delta_{0}$--regular so this term is $\leq (1/3) \abs{\cG}$. The
 second term represents the weights of the remainders, after removing
 $\cl R(I_{j})$, from each $\delta_{0}$--regular $I_{j}$. This sum is
\begin{equation*}
\begin{split} \leq \Ca \sum_{j}w_{j}\leb (I_{j}\setminus
R(I_{j}))/\leb(I_{j}) &\leq \Ca \sum_{j}w_{j}\leb(\ball_{\epstail})/\leb(R(I_{j}))  \\ &\le \Ca \abs{\cG}
\Cball(\epstail)/\leb(\cR) , \end{split}
\end{equation*} where $\leb(\cR) = \min_{1\le k\le N} \leb(R_{k})$ and $\ball(\epstail)$ denotes a ball of radius $\epstail$. So
 the expression in the parentheses and on the right-hand side of
 \eqref{eq:bound2} is $\leq \abs{\cG} (1/3+\Ca
 \Cball(\epstail)/\leb(\cR))$. Therefore the inequality
 \eqref{eq:bound2} is satisfied if we take:
\begin{equation}
 \label{eq:t} t=\frac{(2/3) \abs{\cG}\Caa
 \Cball(\epstail)^{-1}\leb(\cR)}{\abs{\cG} (1/3+\Ca
 \Cball(\epstail)/\leb(\cR))} =\frac{(2/3)\Caa
 \Cball(\epstail)^{-1}\leb(\cR)^{2}}{(1/3)\leb(\cR)+\Ca
 \Cball(\epstail)} .
\end{equation}
\end{proof}

\begin{lem}	\label{lem:fully_rec}
	Suppose $Z$ is fully recurrent at times $\{\tilde n_j\}_{j=1}^K$ and
		$C_1,C_2>0$ are arbitrary constants. Then $Z$ is fully recurrent at times
		$\{n_j\}_{j=1}^K$, where $n_1 \ge C_1$ and $n_{j+1}-n_j \ge C_2$
		for every $j \in \{1, \dots, K-1\}$. 
\end{lem}	
\begin{proof}
	Let $\{m_{j}\}_{j=1}^K \in \bN \cup \{0\}$ be
		s.t. $m_1\tilde n_K \ge C_1$ and $(m_{j+1} -m_j)\tilde n_{1}
		\geq C_2$ and define
		\begin{equation*}
			\begin{split} 
			n_{j}&:=
				\tilde n_{j}+m_{j}\tilde n_{K}, \text{ if } 1 \le j\le K-1;\\
				n_{K}&:=\tilde n_{K}+\sum_{j=1}^{K-1}n_{j}.
			\end{split}
		\end{equation*}
	 It follows from the properties of $\gcd$ that
		$$\gcd(n_1 ,\dots, n_K) = \gcd(\tilde n_1 ,\dots, \tilde n_K).$$
		Also, by definition, $n_1 \ge m_1 \tilde n_K \ge C_1$ and $n_{j+1}-n_j \ge (m_{j+1}-m_j)\tilde n_1 \ge C_2$.
		
		Since $Z$ covers itself when it returns at times $\{\tilde
		n_j\}$, the same holds at times $\{n_j\}$. It follows
		that $Z$ is fully recurrent at times $\{n_{j}\}_{j=1}^{K}$.
\end{proof}

\section{Toolbox}
\label{sec:toolbox}
\subsection{Transfer operator} Define the
\textit{transfer operator} $\sL:\L^1(\uspace, \leb) \circlearrowleft$
as the dual of the Koopman operator $U: \L^\infty(\uspace, \leb)
\circlearrowleft$, $Ug=g \circ T$. By a change of variables, it
follows that
\begin{equation}
 \sL f(x) = \sum_{h \in \cH}f \circ h(x) \cdot Jh(x) \cdot
 \Id_{T(O_h)}(x), \text{ for } \leb \text{-a.e. } x \in \uspace.
\end{equation}

 Note that $\sL^n f(x)= \sum_{h \in \cH^n} f \circ
 h(x)Jh(x)\Id_{T^n(O_h)}(x)$, for every $n \in \bN$.
 \subsection{Standard families} For $\alpha \in (0,1)$, and a function
 $\rho: I \to \bR^+:=(0, \infty)$, $I \subset \uspace$ define
\begin{equation} H(\rho) := H_{\alpha}(\rho) = \sup_{x,y \in I}
 \frac{\abs{\ln \rho(x) - \ln \rho(y)}}{\metr(x,y)^\alpha}.
\end{equation}
 
\begin{defin} [Standard pair] \label{std_pair} An $(a,
\epstail)$--\emph{standard pair} is a pair $(I, \rho)$ consisting of
an open set $I \subset \uspace$ and a function $\rho:I \to \bR^+$ such
that $\diam{I} \le \epstail$, $\int_I \rho =1$ and
\begin{equation}
\label{eq:mod_regularity} H(\rho) \leq a.
\end{equation}
 \end{defin}
\begin{rem} [Notation] All integrals where the measure is not
indicated are with respect to the underlying measure $\leb$. \end{rem}
\begin{defin} [Standard family]\label{std_family} An $(a,
\epstail)$--\emph{standard family} $\cG $ is a set of $(a,
\epstail)$--standard pairs $\{(I_j, \rho_j)\}_{j \in \cJ}$ and an
associated measure $ w_\cG $ on a countable set $\cJ $. The
\emph{total weight} of a standard family is denoted $\abs{\cG}:=
\sum_{j \in \cJ} w_j $. We say that $\cG $ is an $(a, \epstail,
B)$--\emph{proper} standard family if in addition there exists a
constant $ B>0$ such that,
\begin{equation}
 \label{eq:bd_def} \abs{\partial_\ve \cG} := \sum_{j \in \cJ} w_\cG(j)
 \int_{\partial_\ve I_j} \rho_j \leq B \abs{\cG}\ve, \text{ for all }
 \ve<\epstail.
\end{equation} If $w_\cG $ is a probability measure on $\cJ$, then
 $\cG $ is called a \emph{probability standard family}. Note that
 every $(a, \epstail)$--standard family induces an absolutely
 continuous measure on $\uspace$ with the density $\rho_\cG := \sum_{j
 \in \cJ} w_j \rho_j \Id_{I_j} $. We say that two standard families
 $\cG$ and $\tilde \cG$ are \textit{equivalent} if
 $\rho_{\cG}=\rho_{\tilde \cG}$. \end{defin} Next we define what we
 mean by an iterate of a standard family. Given an $(a,
 \epstail)$--standard family $\cG$, we define an $n$-th iterate of
 $\cG$ as follows.
\begin{defin} [Iteration] \label{iteration} Let $\cG$ be an $(a,
 \epstail)$--standard family with index set $\cJ$ and weight
 $w_{\cG}$. For $(j,h) \in \cJ \times \cH^n$ such that $\diam{T^n(I_j
 \cap O_h)} > \epstail$ and for an open set $V_{*} \subset T^n(I_j
 \cap O_h)$, $\diam V_{*}\le \epstail/(4\dimn^{1/2})$ let $\cU_{(j,h)}$ be the
 index set of a \footnote{The existence of such a partition
 $\{U_{\ell}\}$ follows essentially from \cite[Proof of
 Theorem~2.1]{Che1} and \cite[p.~1349]{BT08}, but for the sake of
 completeness it is also shown in \Cref{lem:divisibility}. There may
 be many admissible choices for such ``artificial chopping''. One can
 make different choices at different iterations hence an $n$--th
 iterate of $\cG$ is by no means uniquely defined (and this does not
 cause any problems).} (mod $0$)-partition $\{U_{\ell}\}_{\ell \in
 \cU_{(j,h)}}$ of $T^n(I_j \cap O_h)$ into open sets such that
\begin{equation}
 \label{eq:chop_size} \diam{U_\ell} < \epstail, \forall \ell \in
 \cU_{(j,h)},
\end{equation} $V_{*} \subset U_{\ell}$ for some $\ell \in
 \cU_{{(j,h)}}$ and such that, setting $V=T^n(I_j \cap O_h)$,
\begin{equation}
 \label{eq:chop_complexity} \frac{\sum_{\ell \in \cU_{(j,h)}}
 \leb(h(\partial_{\ve}U_{\ell} \setminus
 \partial_{\ve}V))}{\leb(h(V))} \leq C_{\epstail}\ve, \text{ for every
 } \ve <\epstail.
\end{equation} For $(j,h) \in \cJ \times \cH^n$ such that
 $\diam{T^n(I_j \cap O_h)} \le \epstail$ set $\cU_{(j,h)}=\emptyset$.
 Define
\begin{equation}
\label{eq:j_n} \cJ_{n}:=\{(j,h,\ell) | (j,h)\in \cJ \times \cH^n, \ell
\in \cU_{(j,h)}, \leb(I_j\cap O_h) >0\}.\footnote{When $\cU_{(j,h)} =
\emptyset$, by $(j,h,\ell)$ we mean $(j,h)$.}
\end{equation} For every $j_{n}:=(j, h, \ell) \in \cJ_{n}$, define
 $I_{j_{n}} := T^n(I_{j} \cap O_h) \cap U_\ell$ and $\rho_{j_{n}}:
 I_{j_{n}} \to \bR^{+}$, $\rho_{j_{n}} := \rho_{j} \circ h \cdot Jh
 \cdot z_{j_{ n}}^{-1}$, where $z_{j_{ n}} :=\int_{I_{j_{ n}}}
 \rho_{j} \circ h Jh$. Define $ \cT^{n}\cG := \left\{\left(I_{j_{ n}},
 \rho_{j_{ n}} \right)\right\}_{j_{ n}\in \cJ_{ n}} $ and associate to
 it the measure given by
\begin{equation}
\label{eq:weight_evol} w_{\cT^{n}\cG}(j_{n}) = z_{j_{ n}} w_{\cG}(j).
\end{equation}
 \end{defin}
\begin{rem}[Notation] To simplify notation throughout the rest of the
paper we write $w_{j_{n}}$ for $w_{\cT^{n}\cG}(j_{n})$ and $w_{j}$ for
$w_{\cG}(j)$. \end{rem}
\begin{rem} If $\cG$ is an $(\modreg, \epstail)$--standard family,
then so is $\cT^{n}\cG$ -- a fact that is justified by
\Cref{invariance} of the next section. Comparing the definition of the
transfer operator applied to a density with the definition of
$\cT^{n}\cG$ and the measure associated to it, we see that
\begin{equation}
 \label{eq:connection} \sL^n \rho_\cG = \rho_{\cT^{n}\cG}.
\end{equation} This is the main connection between the evolution of
 densities under $\sL^n$ and the evolution of standard families.
 \end{rem}
\begin{rem} A simple change of variables shows that for every standard
family $\cG$ and every $n\in \bN$, $\abs{\cT^{n}\cG}=\abs{\cG}$. That
is, the total weight does not change under iterations. We will make
use of this fact throughout the article. \end{rem}
\begin{lem} [Artificial chopping avoiding a small set $V_{*}$]
\label{lem:divisibility} Suppose $V$ is a bounded, open subset of
$\bR^{\dimn}$ with $\diam V > \epstail$, $V_{*} \subset V$ is a subset
of $\diam V_{*} \le \epstail/(4\sqrt{\dimn})$ and $T$ satisfies
\ref{hyp1} and \ref{hyp2}. Then, there exists a (mod $0$)-partition
$\{U_{\ell}\}_{\ell \in \cU}$ of $V$ into open sets such that
$\diam{U_\ell} \le \epstail$ $\forall \ell \in \cU$, $V_{*} \subset
U_{\ell}$ for some $\ell \in \cU$, and
\begin{equation}
 \label{eq:divisibilitycond} \frac{\sum_{\ell \in \cU}
 \leb(h(\partial_{\ve}U_{\ell} \setminus
 \partial_{\ve}V))}{\leb(h(V))} \leq C_{\epstail}\ve, \text{ for every
 } \ve <\epstail,
\end{equation} where $C_{\epstail} =
 e^{\dist\diam(\uspace)^{\alpha}}6\dimn^{3/2} \cdot \epstail^{-1}$.
 \end{lem}
 
\begin{proof} $\{U_{\ell}\}$ will be a family of sets formed by
intersecting $V$ with a grid of cubes of side-length
$\epstail/(3\sqrt{\dimn})$. Indeed, following \cite[Proof of
Theorem~2.1]{Che1} and \cite[p.~1349]{BT08}, let
$\epstail'=\epstail/(3\sqrt{\dimn})$ and given $0\le a_{i}<\epstail'$,
$i=1,\dots,d$, consider the $(d-1)$--dimensional families of
hyperplanes:
\begin{equation*} L_{a_{i}}=\{(x_{1},\dots, x_{i},
 a_{i}+n_{i}\epstail', x_{i+1}, \dots, x_{d-1})| n_{i} \in \bZ\}.
\end{equation*} Denote the $(\dimn-1)$--dimensional volume of $V \cap
 L_{a_{i}}$ by $A_{a_{i}}$. By Fubini theorem,
 $\int_{0}^{\epstail'}A_{a_{i}}\ da_{i} = \leb(V)$. Therefore,
 $\exists a_{i}'$ such that $A_{a_{i}'}\le \leb(V)/\epstail'$. Let $L
 = \cup_{i} L_{a_{i}}'$ and denote the total $(\dimn-1)$--dimensional
 volume of $L\cap V$ by $A$. Let $\cS=\{S_{\ell}\}_{\ell \in \cU}$ be
 the collection of cubes of the grid formed by $L$ that intersect $V$.
 Let $U_{\ell}=S_{\ell} \cap V$, $\forall \ell \in \cU$. Then we have
\begin{equation*}
 \leb(\cup_{\ell \in \cU}(\partial_{\ve}U_{\ell}\setminus
 \partial_{\ve}V)) \le 2\ve A \le 2\ve \dimn\leb(V)/\epstail' =
 6\dimn^{3/2}\ve\leb(V).
\end{equation*} Now \eqref{eq:divisibilitycond} follows by using the
 distortion bound \eqref{eq:dist}. Finally, suppose $\diam V_{*} \le
 \epstail/(4\sqrt{\dimn})$. Let $\ell' \in \cU$ be such that
 $S_{\ell'}\cap V_{*}\neq \emptyset$. Consider the $2^{\dimn}+2\dimn$
 elements of $\cS$ that share a face or a vertex with the cube
 $S_{\ell'}$. Denote them by
 $\{S_{\ell_{j}}\}_{j=1}^{2^{\dimn}+2\dimn+1}$ and let $S_{\ell_{*}}=
 \cup_{j=1}^{2^{\dimn}+2\dimn+1}S_{\ell_{j}}$, $U_{\ell_{*}}=
 \cup_{j=1}^{2^{\dimn}+2\dimn+1}U_{\ell_{j}}$. The set $U_{\ell_{*}}$
 is contained in a cube of side-length $\epstail/\sqrt{\dimn}$ hence
 it has diameter $\le \epstail$. Moreover, $U_{\ell_{*}}$ must contain
 the set $V_{*}$ because otherwise $V_{*}$ would contain one point
 from $S_{\ell'}$ and another point from $V \setminus S_{\ell_{*}}$.
 By construction the distance between these points would be greater
 than $\epstail/(3\sqrt{\dimn})$, which contradicts $\diam V_{*} \le
 \epstail/(4\sqrt{\dimn})$. Now, in the collection $\{U_{\ell}\}$,
 replace the elements $\{U_{\ell_{j}}\}_{j=1}^{2^{d}+2d+1}$ with the
 set $U_{\ell_{*}}$. Then $V_{*} \subset U_{\ell_{*}}$, $\diam
 U_{\ell_{*}} \le \epstail$ and condition \eqref{eq:divisibilitycond}
 is still satisfied.
\end{proof} Next, we show the invariance of standard families under
 iteration, but first we state a simple lemma that provides a useful
 consequence of log-H\"older regularity \eqref{eq:mod_regularity}.
\begin{lem} [Comparability Lemma] \label{Fed} If $\rho: I \to \bR^{+}$
 satisfies $H(\rho) \leq a$ for some $a\geq 0$ and $\diam I \le
 \epstail$, then for every $J, J' \subset I$ with $\leb(J)\leb(J')\neq
 0$,
\begin{equation}
 \label{eq:comp1} \inf_I \rho \asymp_{a} \cA_J \rho \asymp_{a} \cA_
 {J'} \rho \asymp_{a} \sup_I \rho,
\end{equation} where $\cA_J \rho = \leb(J)^{-1} \int_J \rho$ is the
 average of $\rho$ on $J$ and $C_{1} \asymp_{a} C_{2}$ means
 $e^{-a\epstail^{\distexp}} C_{1}\leq C_{2}\leq
 e^{a\epstail^{\distexp}} C_{1} $. \end{lem}
 
\begin{proof} The lemma follows from the fact that if $(I, \rho)$
satisfies $H(\rho) \leq a$, then for every $x, y \in I$,
\begin{equation*} e^{-a\metr(x,y)^{\distexp}} \rho(y)\leq \rho(x)\leq
e^{a\metr(x,y)^{\distexp}} \rho(y).
\end{equation*}
 
\end{proof} The following lemma together with \Cref{iteration} justify
 the invariance of an $(\modreg, \epstail)$--standard family under
 iteration.
\begin{lem}
\label{invariance} Suppose $(I, \rho)$ is an $(\modreg,
\epstail)$--standard pair and $(I_n, \rho_n)$ is an image of it under
$T^n$ for some $n \in \bN$, as in \Cref{iteration}. Then
$\diam(I_{n})\le \epstail$, $\int_{I_n} \rho_{n} = 1$ and
\begin{equation}
 \label{eq:mod_reg_1} H(\rho_n) \leq \modreg (\expan^{\distexp n} +
 \modreg^{-1}\dist).
\end{equation}
 \end{lem}
 
\begin{proof} Using the definition of $H(\cdot)$, noting its
properties under multiplication and composition, and using the
expansion of the map, it follows that Let us start by stating
\begin{equation*} H(\rho_{j_{n}}) \leq H(Jh) + \expan^{\distexp
n}H(\rho_{j}).
\end{equation*} By \eqref{eq:dist} we have $H(Jh) \leq \dist$, and by
 assumption $H(\rho_{j})\leq \modreg$, finishing the proof of
 \eqref{eq:mod_reg_1}.
\end{proof} The next lemma plays a crucial role in our arguments
 because it allows us to control the measure of points that map near
 the discontinuities.
\begin{lem} [Growth Lemma] \label{p_growth} Suppose $\epstail>0$,
$\ntail\in\bN$ and $\sigtail$ are as in our assumptions. Suppose $\cG$
is an $(\modreg, \epstail)$--standard family. Then for every $\ve <
\epstail$ we have
\begin{equation}
\label{eq:growth_lemma} \abs{\partial_\ve \cT^{\ntail}\cG} \leq (1+\Ca
\sigtail)\abs{\partial_{\expan^{\ntail}\ve}\cG}+\zeta_1 \abs{\cG}\ve,
\end{equation} where $\zeta_1= \Ca C_{\epstail}$. \end{lem}
 
\begin{proof} Suppose $\ve < \epstail$. We write $n$ for $\ntail$. We
have, by definition, $\abs{\partial_\ve \cT^{n}\cG} = \sum_{j_{n}}
w_{j_n}\int_{\partial_\ve I_{j_{n}}} \rho_{j_{n}} $. We split the sum
into two parts according to whether $\cU_{(j,h)} = \emptyset$ or
$\cU_{(j,h)} \neq \emptyset$. Suppose $\cU_{(j,h)} = \emptyset$, that
is $\diam T^{n}(I_{j} \cap O_{h}) \le \epstail$ and $I_{j_{n}}=
T^{n}(I_{j} \cap O_{h}) $. By a change of variables,
\begin{equation*} w_{j_n}\int_{\partial_\ve I_{j_{n}}} \rho_{j_{n}} =
 w_j \int_{h(\partial_\ve I_{j_n})}\rho_{j}.
\end{equation*} For every $h \in \cH^{n}$, since
 $h(\partial_{\ve}I_{j_{n}}) \subset O_{h}$, we can write
\begin{equation}
 \label{eq:domainsplit} h(\partial_\ve I_{j_n}) \subset
 \left(h(\partial_{\ve}I_{j_{n}})\setminus
 \partial_{\expan^{n}\ve}I_{j}\right) \cup
 (\partial_{\expan^{n}\ve}I_{j}\cap O_{h}).
\end{equation} The integral over $\partial_{\expan^{n} \ve}I_{j} \cap
 O_{h}$, and summed up over $h$ and $j$ is easily estimated by
 $\abs{\partial_{\expan^{n} \ve}\cG}$. To estimate the integral of
 $\rho_{j}$ over $h(\partial_{\ve}I_{j_{n}})\setminus
 \partial_{\expan^{n}\ve}I_{j}$ we compare it, using \Cref{Fed}, to
 $\int_{\partial_{\expan^{n}\ve}I_{j}} \rho_{j}$ and we get
\begin{equation*}
 \int_{h(\partial_{\ve}I_{j_{n}})\setminus
 \partial_{\expan^{n}\ve}I_{j}} \rho_{j} \leq \Ca
 \frac{\leb(h(\partial_{\ve}T^{n}(I_{j} \cap O_{h}))\setminus
 \partial_{\expan^{n}\ve}I_{j})}{\leb(\partial_{\expan^{n}\ve}I_{j})}
 \int_{\partial_{\expan^{n}\ve}I_{j}} \rho_{j}
\end{equation*} Note that if $\leb(I_{j} \cap O_{h})=0$, then
 $\leb(h(\partial_{\ve}T^{n}(I_{j} \cap O_{h}))) =0$ since $
 h(\partial_{\ve}T^{n}(I_{j} \cap O_{h})) \subset I_{j} \cap O_{h} $.
 By the controlled complexity condition \eqref{eq:dyncomplexity},
\begin{equation}
 \label{eq:boundthefraction} \sum_{h \in \cH^{n}}
 \frac{\leb(h(\partial_{\ve}T^{n}(I_{j} \cap O_{h}))\setminus
 \partial_{\expan^{n}\ve}I_{j})}{\leb(\partial_{\expan^{n}\ve} I_{j})}
 \leq\sigtail.
\end{equation} Therefore,
\begin{equation} \label{eq:endest}
 \sum_{j \in \cJ}w_{j}\sum_{h \in
 \cH^{n}}\int_{h(\partial_{\ve}I_{j_{n}})\setminus
 \partial_{\expan^{n}\ve}I_{j}} \rho_{j} \leq \Ca \sigtail
 \abs{\partial_{\expan^{n} \ve}\cG}.
\end{equation} Now suppose that $\cU_{(j,h)} \neq \emptyset$. By
 \Cref{iteration}, $ \sum_{j_{n}}w_{j_n} \int_{\partial_\ve I_{j_n}}
 \rho_{j_n} $ is bounded by $\leq \sum_j w_j \sum_{h, \ell}
 \int_{\partial_\ve I_{j_n}} \rho_j \circ h Jh $. Let us split the
 integral over two sets. Since $\partial_{\ve}I_{j_{n}} \subset
 U_{\ell}$, we can write
\begin{equation}
 \label{eq:anothersplit} \partial_{\ve}I_{j_{n}} \subset (\partial_\ve
 I_{j_n}\setminus \partial_{\ve}T^{n}(I_{j}\cap O_{h})) \cup
 (\partial_{\ve}T^{n}(I_{j}\cap O_{h}) \cap U_{\ell}).
\end{equation} Consider the first term on the right-hand side of
 \eqref{eq:anothersplit}. We need to estimate the integral of
 $\rho_{j} \circ h Jh$ on this set and sum over $\ell$, $h$ and $j$.
 Using a change of variables, the integral is
\begin{equation*}
 \int_{h(\partial_\ve I_{j_n}\setminus \partial_{\ve}T^{n}(I_{j}\cap
 O_{h}))} \rho_{j}.
\end{equation*} Since $H(\rho_{j}) \leq \modreg$, $h(\partial_\ve
 I_{j_n}\setminus \partial_{\ve}T^{n}(I_{j}\cap O_{h})) \le
 \diam(I_{j}) \le \epstail$ and $\diam (h(T^{n}(I_{j}\cap O_{h}))) =
 \diam (I_{j} \cap O_{h})\le \diam(I_{j}) \le \epstail $, we apply
 \Cref{Fed} to get
\begin{equation*}
 \int_{h(\partial_\ve I_{j_n}\setminus \partial_{\ve}T^{n}(I_{j}\cap
 O_{h}))} \rho_{j} \le \Ca \frac{\leb(h(\partial_\ve I_{j_n}\setminus
 \partial_{\ve}T^{n}(I_{j}\cap O_{h})))}{\leb(h(T^{n}(I_{j}\cap
 O_{h})))}\int_{h(T^{n}(I_{j}\cap O_{h}))} \rho_j
\end{equation*} Now we sum the above expression over $\ell$, which is
 implicit in the notation $I_{j_{n}}=T^{n}(I_{j}\cap O_{h})\cap
 U_{\ell}$. Using \eqref{eq:chop_complexity}, we get
\begin{equation*}
 \leq \Ca C_{\epstail}\ve \int_{I_{j} \cap O_{h}}\rho_{j}
\end{equation*} Now we sum over $h$, multiply by $w_{j}$ and sum over
 $j$. As a result we get the estimate $\le \Ca C_{\epstail}\ve
 \abs{\cG}$. Consider the second term on the right-hand side of
 \eqref{eq:anothersplit}. The contribution of this set is equal to
 $\sum_{j}w_{j}\sum_{h}\int_{h(\partial_{\ve}T^{n}(I_{j}\cap
 O_{h}))}\rho_{j}$. But this was already included in the upper-bound estimate 
 above \eqref{eq:domainsplit}-\eqref{eq:endest}, so we do not need to add
 it again.
\end{proof} Recall from \Cref{sec:setting} that $\ntail$ is such that
 $\vartheta_1:=\expan^{\ntail}(1+\Ca \sigtail)<1$. Iterating \Cref{p_growth} leads
 to the following. The proof is standard (uses
 \eqref{eq:dyncomplexitybound}) so we omit it.
\begin{cor}\label{iteratedgrowthlemma} There exists $\zeta_{2}\ge 0$
such that for every $k \in \bN$ and $\ve < \epstail$,
\begin{equation}
 \label{eq:iteratedgrowthlemma} \abs{\partial_{\ve}\cT^{k\ntail}\cG}
 \leq (1+\Ca \sigtail)^{k} \abs{\partial_{\expan^{k\ntail}\ve}\cG} +
 \zeta_{2}\abs{\cG}\ve.
\end{equation} Moreover, there exist $\zeta_{2}, \zeta_{3} \ge 0$ such
 that for every $m \in \bN$ that does not divide $\ntail$ and for
 every $\ve < \epstail$,
\begin{equation}
 \abs{\partial_{\ve}\cT^{m}\cG} \leq \zeta_{3}(1+\Ca
 \sigtail)^{m/\ntail} \abs{\partial_{\expan^{m}\ve}\cG} + \zeta_{4}
 \abs{\cG}\ve .
\end{equation}
 \end{cor}
\begin{prop}
\label{prop:bd_invariance} There exists $\proper >0$ such that for every $B
\ge \proper$ there exists $n_{rec}(B) \in \bN$ such that if $\cG$ is
an $(\modreg, \epstail, B)$--proper standard family, then for every $m
\ge n_{rec}(B)$, $\cT^m\cG$ is an $(\modreg, \epstail,
\proper)$--proper standard family. \end{prop}
 
\begin{proof}
Choose $\proper > \zeta_4$. Seting
$\vartheta_{2}=\vartheta_{1}^{1/\ntail}<1$ it follows from
\Cref{iteratedgrowthlemma} that for every $m \in \bN$ and $\ve
<\epstail $,
\begin{equation}
 \label{eq:bd_inv} \abs{\partial_\ve \cT^{m}\cG} \leq \abs{\cG}\ve
 (B\zeta_{3}\vartheta_2^{m} +\zeta_4).
\end{equation} Now choose $n_{rec}(B)$ so that
 $B\zeta_{3}\vartheta_2^{n_{rec}} +\zeta_4\le\proper$.
\end{proof}
 
\begin{rem} \label{rem:Brecov} $n_{rec}:[0, \infty) \to \bN$ denotes
the time it takes for an $(\modreg, \epstail, \cdot)$--proper standard
family to recover to an $(\modreg, \epstail, \proper)$--proper
standard family. \end{rem}
\begin{defin}
\label{deltareg} A set $I \subset \uspace$ is said to be
$\delta$--\emph{regular} if $I$ is open and $\leb(I\setminus
\partial_{\delta}I) >0$.
\end{defin}
\begin{lem} If $I$ is a $\delta$-regular set, then for every $x \in
 I\setminus \partial_{\delta}I $, the ball $\ball(x, \delta)$ is
 contained in $I$.
 \end{lem}

\begin{proof}
If $I$ is $\delta$-regular, then $I\setminus \partial_\delta I$ is
non-empty. Consider a point $x \in I\setminus \partial_\delta I$ and
the ball (in $\bR^\dimn$) $\ball(x,\delta)$ of radius $\delta$
centered at $x$. If this ball is not entirely contained in $I$, then
$\ball(x,\delta) \cap I$ and $\ball(x,\delta)\cap (\bR^\dimn\setminus
\cl I)$ are non-empty open sets in $\bR^\dimn$. Furthermore since $I$
is open and $\ball(x,\delta)$ does not intersect $\partial I$, the
union of the sets $\ball(x,\delta)\cap (\bR^\dimn\setminus \cl I)$ is
$\ball(x,\delta)$. This is a contradiction to $\ball(x,\delta)$ being
connected in $\bR^\dimn$.
\end{proof}
 
\begin{rem}\label{rem:delzero} Define $\delta_{0}=1/(3\proper)$. It follows that if $\cG$
is an $(\modreg, \epstail, B)$--proper standard family, then more than
$(2/3)$ of its total weight is concentrated on $\delta_{0}$-regular
sets. That is,
\begin{equation}
 \sum_{j\in \cJ_{reg}}w_{j} \ge \sum_{j} w_{j}\int_{I_{j}\setminus
 \partial_{\delta_{0}}I_{j}} \rho_{j} \ge 2/3,
\end{equation} where $\cJ_{reg}$ corresponds to indices $j$ for which
 $I_{j}$ is $\delta_{0}$-regular. \end{rem}

\end{document}